\documentclass[11pt]{amsart}
\usepackage{verbatim,amssymb,amsmath,graphicx,hyperref}

  \scrollmode

\newcommand{\R}{{\mathbb R}}
\newcommand{\N}{{\mathbb N}}

\newcommand{\EE}{{\mathbb E}}
\newcommand{\PP}{{\mathbb P}}

\newcommand{\eul}{{\widehat X}}
\newcommand{\eultr}{{\widehat Z}}
\newcommand{\ind}{1}
\newcommand{\usn}{\underline {s}_n}
\newcommand{\utn}{\underline {t}_n}

\newcommand{\sgn}{\operatorname{sgn}}
\newcommand{\eps}{\varepsilon}

\newcommand{\F}{{\mathcal F}}

\theoremstyle{plain}
\newtheorem{theorem}{Theorem}
\newtheorem{prop}{Proposition}
\newtheorem{lemma}{Lemma}

\theoremstyle{definition}
\newtheorem{rem}{Remark}

\begin{document}
\title[]{On the performance of the Euler-Maruyama scheme  for 
SDEs with discontinuous drift coefficient}

\author[M\"uller-Gronbach]
{Thomas M\"uller-Gronbach}
\address{
Faculty of Computer Science and Mathematics\\
University of Passau\\
Innstrasse 33 \\
94032 Passau\\
Germany} \email{thomas.mueller-gronbach@uni-passau.de}

\author[Yaroslavtseva]
{Larisa Yaroslavtseva}
\address{
Seminar for Applied Mathematics\\
Department of Mathematics\\
HG G 53.2\\
R\"amistrasse 101 \\
8092 Zurich\\
Switzerland} \email{larisa.yaroslavtseva@sam.math.ethz.ch}

\begin{abstract}
Recently a lot of effort has been invested to analyze the 
$L_p$-error 
 of the Euler-Maruyama scheme in the case of stochastic differential equations (SDEs) with a drift coefficient that may have discontinuities in space. For scalar SDEs with a piecewise Lipschitz drift coefficient and a Lipschitz diffusion coefficient that is non-zero at the discontinuity points  of the drift coefficient so far only 
an
$L_p$-error rate 
of at least
$1/(2p)-$
has been proven.
In the present paper we show that under the latter  conditions
  on the coefficients of the SDE
 the Euler-Maruyama scheme in fact achieves
an $L_p$-error rate of at least $1/2$ for all $p\in [1,\infty)$
as in the case of SDEs with Lipschitz coefficients.
\end{abstract}
\maketitle

\section{Introduction}
Consider an autonomous stochastic differential equation (SDE)
\begin{equation}\label{sde000}
\begin{aligned}
dX_t & = \mu(X_t) \, dt + \sigma(X_t) \, dW_t, \quad t\in [0,1],\\
X_0 & = x_0
\end{aligned}
\end{equation}
with deterministic initial value $x_0\in\R$,  drift coefficient $\mu\colon\R\to\R$,  diffusion coefficient $\sigma\colon \R\to\R$ and  $1$-dimensional driving Brownian motion $W$. 
If~\eqref{sde000} has a unique strong solution $X$ then
a classical numerical approach
for approximating 
$X_1$  based on $n$ observations of $W$  is provided by
the Euler-Maruyama   scheme  given  by
$\eul_{n,0}=x_0$ and
\[
\eul_{n,(i+1)/n}=\eul_{n,i/n}+\mu(\eul_{n,i/n})\cdot 1/n+\sigma(\eul_{n, i/n})\cdot (W_{(i+1)/n}-W_{i/n})
\]
for $i\in\{0,\ldots, n-1\}$.

It is well-known that if the coefficients $\mu$ and $\sigma$ are  Lipschitz continuous then 
for all $p \in [1,\infty)$ the Euler-Maruyama scheme
at the final time
achieves 
an 
$L_p$-error rate 
   of  
at least $1/2$ in terms of the number $n$ of observations of $W$, 
i.e. 
for all $p \in [1,\infty)$ there exists $c\in (0,\infty)$ such that 
for all $n\in\N$,
\begin{equation}\label{result}
\bigl(\EE\bigl[ |X_1-\eul_{n,1}|^p\bigr]\bigr)^{1/p}\leq \frac{c}{\sqrt n}.
\end{equation}

In this article we study the $L_p$-error 
of
 $\eul_{n,1}$  in the case when the drift coefficient $\mu$ may have finitely many discontinuity points. More precisely, we assume that 
 the drift coefficient
 $\mu$ is piecewise Lipschitz continuous in the sense that
\begin{itemize}
\item[(A1)] there exist $k\in\N_0$ and $\xi_0, \ldots, \xi_{k+1}\in [-\infty,\infty]$ with $-\infty=\xi_0<\xi_1<\ldots < \xi_k <\xi_{k+1}=\infty$ such that
 $\mu$ is Lipschitz continuous on the interval $(\xi_{i-1}, \xi_i)$ for all $i\in\{1, \ldots, k+1\}$,
\end{itemize}
and we assume that the diffusion coefficient  $\sigma$  is Lipschitz continuous and  
non-zero
 at the  potential discontinuity points of $\mu$, i.e. 
\begin{itemize}
\item[(A2)] $\sigma$ is Lipschitz continuous on $\R$ and $\sigma(\xi_i) \neq 0$ for all $i\in\{1,\ldots,k\}$.
\end{itemize}
Note that under the assumptions
(A1) and (A2) the equation~\eqref{sde000} has a unique strong solution, see ~\cite[Theorem 2.2]{LS16}.

Numerical   approximation  of SDEs with 
a 
drift coefficient that is discontinuous in space has gained a lot of interest in recent years, see~\cite{ g98b, gk96b} for results on convergence in probability and  almost sure convergence  of the Euler-Maruyama scheme  and~\cite{ GLN17, HalidiasKloeden2008,  LS16,  LS15b, LS18, NSS18,Tag16, Tag2017b, Tag2017a} for results on $L_p$-approximation. 
In particular, in~\cite{LS18, Tag16, Tag2017b, Tag2017a} the $L_p$-error 
of the Euler-Maruyama  scheme has been studied for such SDEs.
The most far going results in the latter four articles provide for the one-dimensional SDE \eqref{sde000} under the assumptions (A1) and (A2)
\begin{itemize}
\item[(i)] an $L_1$-error rate of at least $1/2$ for $\eul_{n,1}$  if, additionally to (A1) and (A2), the coefficients $\mu$ and $\sigma$ are bounded, $\mu$ is integrable on $\R$ or one-sided Lipschitz continuous, and $\sigma$ is bounded away from zero, see~\cite{Tag16, Tag2017b},
\item[(ii)] an $L_1$-error rate  of at least $1/2-$ for $\eul_{n,1}$ if, additionally to (A1) and (A2), the coefficients  $\mu$ and $\sigma$ are bounded and $\sigma$ is bounded away from zero, see \cite{Tag2017b},
\item[(iii)] an $L_2$-error rate of  at least $1/4-$ for $\eul_{n,1}$, if, additionally to (A1) and (A2), the coefficients $\mu$ and $\sigma$ are bounded, see ~\cite{LS18}.
\end{itemize}
We add that the proof techniques in~\cite{LS18} can readily be adapted to show that  the Euler-Maruyama scheme
at the final time $\eul_{n,1}$ achieves an
$L_p$-error rate of at least $1/(2p)-$ for all $p\in[1, \infty)$ if the coefficients $\mu$ and $\sigma$ are bounded and satisfy the assumptions (A1) and (A2), see the discussion at the beginning of Section \ref{Proofs}.
Furthermore, in~\cite[Remark 4.2]{NSS18} it is stated that the proof techniques in~\cite{LS18} could be modified to cover the case of unbounded coefficients $\mu$ and $\sigma$ as well.

To summarize, under the assumptions (A1) and (A2)  it was only known up to now that  
the Euler-Maruyama scheme 
at the final time  
achieves 
an $L_p$-error rate of at least $1/(2p)-$
for all $p\in[1, \infty)$,
and it was a challenging question whether  
these error bounds 
can be improved, and if so, whether under the assumptions (A1) and (A2) the 
Euler-Maruyama scheme 
at the final time
even achieves an $L_p$-error rate of at least $1/2$ for all $p\in[1, \infty)$ as it is the case for SDEs with Lipschitz continuous coefficients, see~\eqref{result}.

Note that the recent literature on numerical approximation of  SDEs contains a number of examples  of SDEs with coefficients that are not Lipschitz continuous 
and such that
the Euler-Maruyama scheme at the final time does not achieve
an $L_p$-error rate of $1/2$,
see \cite{ GJS17, hhj12, HefJ17, hjk11, JMGY15, MGRY2018, Y17}. Furthermore, in~\cite{GLN17} numerical studies are carried out for a number of SDEs~\eqref{sde000} with 
a discontinuous 
$\mu$ satisfying (A1) and $\sigma=1$, and for several of these 
SDEs
an empirical $L_2$-error rate  significantly smaller than $1/2$ is observed for  
the Euler-Maruyama scheme at the final time.

However,  regardless of the latter negative findings it turns out that under the 
assumptions
(A1) and (A2)  the Euler-Maruyama 
scheme at the final time
$\eul_{n,1}$ 
in fact satisfies~\eqref{result} for all $p\in [1,\infty)$. This estimate is an immediate consequence of our main result, Theorem \ref{Thm1}, which states that under the assumptions (A1) and (A2) the maximum error of the time-continuous Euler-Maruyama scheme 
achieves at least the rate
$1/2$  in 
the
$p$-th mean sense, for all $p\in[1,\infty)$, see Section~\ref{s3}.

We add that in~\cite{LS16, LS15b} a numerical 
method for approximating $X_1$
is constructed that is based on a suitable transformation of the solution $X$ of~\eqref{sde000} and achieves an $L_2$-error rate of at least $1/2$ in terms of the number of observations of $W$ under the assumptions (A1) and (A2). Furthermore, in~\cite{NSS18} an adaptive Euler-Maruyama scheme is constructed, which achieves 
at the final time
an $L_2$-error rate of at least $1/2-$ in terms of the average number of observations of $W$ under the assumptions (A1) and (A2). However, in contrast to the classical Euler-Maruyama scheme, an implementation of 
either of 
the latter two methods  requires the knowledge of the points of discontinuity of $\mu$.

In this 
paper we furthermore consider the piecewise linear interpolation $\overline{X}_n =(\overline X_{n,t})_{t\in [0,1]}$ of the Euler-Maruyama scheme $(\widehat X_{n,i/n})_{i=0,\dots,n}$
and we study the performance 
 of $\overline{X}_n$ globally on $[0,1]$.  Using Theorem~\ref{Thm1} we show that if the assumptions (A1) and (A2) are satisfied then 
for all $p\in [1,\infty)$ and all $q\in [1,\infty]$  there exists $c\in(0, \infty)$ such that for all $n\in\N$, 
\begin{equation}\label{ll4}
\bigl(\EE\bigl[\|X-\overline X_n\|_q^p\bigr]\bigr)^{1/p}\leq \begin{cases} 
c/\sqrt n, & \text{ if }q < \infty,\\ c \sqrt{\ln (n+1)}/\sqrt{n}, & \text{ if }q = \infty,
\end{cases} 
\end{equation}
where $\|\cdot\|_q$ denotes the $L_q$-norm on the space of real-valued, continuous functions on $[0,1]$, see Theorem~\ref{Thm2}.

Our results provide upper error bounds for the Euler-Maruyama scheme at the final time $\widehat X_{n,1}$ and the piecewise linear interpolation $\overline X_n$ of the Euler-Maruyama scheme 
in terms of the number $n$ of observations of the driving Brownian motion $W$ that are used. It is natural to ask whether these bounds are asymptotically sharp or whether 
there exist alternative algorithms 
based on $n$ observations of $W$
that achieve under the assumptions (A1) and (A2) better rates of convergence in terms of the number $n$. For the error criteria considered in~\eqref{ll4} the answer to this question is already known. The corresponding error rates can not be improved in general, see~\cite{HHMG18, hmr01, MG02_habil} for the case $q\in [1,\infty)$ and~\cite{HHMG18,m02} for the case $q=\infty$. For 
the
$L_p$-approximation of $X_1$ the question is open up to now.
For this problem it is so far only known that under the assumptions (A1) and (A2) it is impossible to obtain 
an $L_p$-error rate
better than $1$ in general, see~\cite{HHMG18,m04}. Whether or not there exists an algorithm that 
approximates $X_1$  
under the assumptions (A1) and (A2)
with an $L_p$-error rate
 better 
 than $1/2$ in terms of the number of observations of $W$ remains a challenging question.  

In the present paper we have only studied scalar SDEs 
while the results 
in~\cite{LS15b, LS18, NSS18, Tag16} also cover the case of multidimensional SDEs.
We believe however that our proof techniques can be extended to obtain for all $p\in[1,\infty)$ an $L_p$-error rate of at least $1/2$ 
for the Euler-Maruyama scheme at the final time in a 
suitable
multidimensional setting as well. This will be the subject of future work.

We briefly describe the content of the paper. Our error estimates, Theorem~\ref{Thm1} and Theorem~\ref{Thm2}, are stated in Section~\ref{s3}. Section~\ref{s4} contains 
proofs of these results and 
a discussion on
the relation of our analysis and the analysis of the Euler-Maruyama scheme carried out in~\cite{LS18}.

\section{Error estimates for the Euler-Maruyama scheme}\label{s3}

Let
$ ( \Omega, \mathcal{F}, \PP ) $ 
be a probability space with a normal filtration
$ ( \mathcal{F}_t )_{ t \in [0,1] } $,
let
$
  W \colon [0,1] \times \Omega \to \R
$
be an 
$ ( \mathcal{F}_t )_{ t \in [0,1] } $-Brownian motion
on $ ( \Omega, \mathcal{F}, \PP ) $, let $x_0\in\R$ and let $\mu, \sigma\colon\R\to\R$ be functions that satisfy the following two conditions.
\begin{itemize}
\item[(A1)] There exist $k\in\N_0$ and $\xi_0, \ldots, \xi_{k+1}\in [-\infty,\infty]$ with $-\infty=\xi_0<\xi_1<\ldots < \xi_k <\xi_{k+1}=\infty$ such that
 $\mu$ is Lipschitz continuous on the interval $(\xi_{i-1}, \xi_i)$ for all $i\in\{1, \ldots, k+1\}$,
\item[(A2)] $\sigma$ is Lipschitz continuous on $\R$ and $\sigma(\xi_i) \neq 0$ for all $i\in\{1,\ldots,k\}$.
\end{itemize}

We consider the SDE
\begin{equation}\label{sde0}
\begin{aligned}
dX_t & = \mu(X_t) \, dt + \sigma(X_t) \, dW_t, \quad t\in [0,1],\\
X_0 & = x_0,
\end{aligned}
\end{equation}
which has a unique strong solution,
see~\cite[Theorem 2.2]{LS16}.

\begin{rem}
Note that if in (A2) the assumption $\sigma(\xi_i) \neq 0$ for all $i\in\{1,\ldots,k\}$ is violated then the existence of a strong solution of \eqref{sde0} can not be guaranteed anymore, see ~\cite[Example 4.2]{LTS15}.
\end{rem}

For $n\in\N$ let $\eul_{n}=(\eul_{n,t})_{t\in[0,1]}$  denote  the time-continuous Euler-Maruyama scheme with step-size $1/n$ associated to the SDE \eqref{sde0}, i.e. $\eul_{n}$ is recursively given by
$\eul_{n,0}=x_0$ and
\[
\eul_{n,t}=\eul_{n,i/n}+\mu(\eul_{n,i/n})\cdot (t-i/n)+\sigma(\eul_{n,i/n})\cdot (W_t-W_{i/n})
\]
for $t\in (i/n,(i+1)/n]$ 
and $i\in\{0,\ldots,n-1\}$. We have the following error estimates for $\eul_{n}$.

\begin{theorem}\label{Thm1}
Let $p\in [1,\infty)$. Then
there exists $c\in(0, \infty)$ such that for all $n\in\N$, 
\begin{equation}\label{l3}
\bigl(\EE\bigl[\|X-\eul_{n}\|_\infty^p\bigr]\bigr)^{1/p}\leq \frac{c}{\sqrt n}. 
\end{equation}
\end{theorem}

Next, we study the performance of the piecewise linear interpolation 
$\overline{X}_n = (\overline X_{n,t})_{t\in [0,1]}$ 
of the time-discrete Euler-Maruyama scheme $(\widehat X_{n,i/n})_{i=0,\dots,n}$, i.e. 
\[
\overline X_{n,t} = (n\cdot t-i)\cdot\widehat X_{n,(i+1)/n} + (i+1-n\cdot t)\cdot \widehat X_{n,i/n}
\]
for $t\in [i/n,(i+1)/n]$
 and $i\in\{0,\ldots,n-1\}$. We  have the following error estimates for $\overline X_n$.

\begin{theorem}\label{Thm2}
Let $p\in [1,\infty)$ and $q\in [1,\infty]$. Then   
there exists $c\in(0, \infty)$ such that for all $n\in\N$, 
\begin{equation}\label{l4}
\bigl(\EE\bigl[\|X-\overline X_n\|_q^p\bigr]\bigr)^{1/p}\leq \begin{cases} 
c/\sqrt n, & \text{ if }q < \infty,\\ c \sqrt{\ln (n+1)}/\sqrt{n}, & \text{ if }q = \infty.
\end{cases} 
\end{equation}
\end{theorem}

\section{Proofs}\label{Proofs}\label{s4}
Throughout this section we put
\[
\utn = \lfloor n\cdot t\rfloor / n
\]
for every $n\in\N$ and every $t\in [0,1]$.

We briefly describe the structure of the proof of our main result, Theorem~\ref{Thm1}, and the relation of our analysis and the analysis of the Euler-Maruyama scheme carried out in~\cite{LS18}. 
Let  $p\in [1,\infty)$. 
In~\cite{LS18} a bijection $G\colon\R\to\R$ is constructed such that $G^{-1}$ is Lipschitz continuous and the stochastic process $Z=G\circ X$ is the unique strong solution of an SDE with Lipschitz continuous coefficients. It then follows by standard error estimates for the Euler-Maruyama scheme that 
there exist $c_1,c_2\in (0,\infty)$ such that for all $n\in\N$,
\begin{equation}\label{sk1}
\begin{aligned}
\bigl(\EE\bigl[\|X-\eul_n\|_\infty\|^p\bigr]\bigr)^{1/p} & \le c_1\cdot \bigl( \EE[\|Z-G\circ\eul_n\|_\infty^p]\bigr)^{1/p}\\
& \le c_2/\sqrt{n} +c_1\cdot \bigl(\EE\bigl[\|\widehat Z_n -G\circ\eul_n\|_\infty^p\bigr]\bigr)^{1/p},
\end{aligned}
\end{equation}
where $\widehat Z_n$ is the time-continous Euler-Maruyama scheme with step-size $1/n$ 
associated to the SDE
for the stochastic process $Z$. 
Using further regularity properties of the function $G$ it is shown in~\cite{LS18} that 
there exists $c\in (0,\infty)$ such that for all $n\in\N$,
\begin{equation}\label{sk2}
\bigl(\EE\bigl[\|\widehat Z_n -G\circ\eul_n\|_\infty^p\bigr]\bigr)^{1/p} \le c/\sqrt{n} + c\cdot \Bigl(\EE\Bigl[\Bigl|\int_0^1 \ind_B (\eul_{n,t},\eul_{n,\utn})\, dt\Bigr|^p\Bigr]\Bigr)^{1/p},
\end{equation}
where 
\[
B= \Bigl(\bigcup_{i=1}^{k+1} (\xi_{i-1},\xi_i)^2\Bigr)^c
\]
is the set of pairs
$(x,y)$ in $\R^2$, which do not allow 
for a joint Lipschitz estimate of $ |\mu(x) -\mu(y)|$ if $\mu$ has at least one discontinuity.
  Finally, using a large deviation argument it is shown in~\cite{LS18} that for every arbitrary small $\delta\in (0,1)$
there exists  $c\in (0,\infty)$ such that for all $n\in\N$,
\begin{equation}\label{sk3}
\Bigl(\EE\Bigl[\Bigl|\int_0^1 \ind_B (\eul_{n,t},\eul_{n,\utn})\, dt\Bigr|^p\Bigr]\Bigr)^{1/p} \le c\cdot n^{-(1-\delta)/(2p)}.
\end{equation}
Combining\eqref{sk1} to~\eqref{sk3} yields
the rate of convergence $1/(2p)-$ 
for the $p$-th root of the $p$-th mean of the maximum error of the time-continuous Euler-Maruyama scheme.

We add that in~\cite{LS18} it is assumed that the coefficients $\mu$ and $\sigma$ are bounded and the analysis is carried out only for $p=2$. However, it is straightforward to adapt the proof technique to the case of a general $p\in[1,\infty)$, and in~\cite[Remark 4.2]{NSS18} it is stated that the proof techniques in~\cite{LS18} could be modified to cover the case of unbounded coefficients $\mu$ and $\sigma$ as well. 

Our proof of Theorem~\ref{Thm1} follows the steps~\eqref{sk1} and~\eqref{sk2} but provides a much better estimate of the $p$-th mean occupation time of the set $B$ than~\eqref{sk3}, namely   
\begin{equation}\label{sk4}
\Bigl(\EE\Bigl[\Bigl|\int_0^1 \ind_B (\eul_{n,t},\eul_{n,\utn})\, dt\Bigr|^p\Bigr]\Bigr)^{1/p} \le c/\sqrt{n},
\end{equation}
which jointly with~\eqref{sk1} and~\eqref{sk2} yields the statement of Theorem~\ref{Thm1}. The estimate~\eqref{sk4} is, essentially, obtained by employing the Markov property of the time-continuous Euler-Maruyama scheme $\eul_n$ relative to the corresponding grid points $1/n, 2/n,\ldots ,1$,  by using appropriate estimates of the expected occupation time of a neighborhood of a non-zero $\xi\in\R$ of $\sigma$ by $\eul_n$ and by carrying out a 
detailed
analysis of the probability of a sign change of $\eul_{n,t}-\xi$ relative to the sign of  $\eul_{n,\utn}-\xi$. 

We briefly describe the 
structure
of this section. In Subsection~\ref{4.1} we provide $L_p$-estimates
of the solution $X$ and the time-continuous Euler-Maruyama scheme $\eul_n$. Subsection~\ref{4.2} provides the Markov property of $\eul_n$ and occupation time estimates for  $\eul_n$, which finally lead to the proof of the estimate~\eqref{sk4}, see Proposition~\ref{prop1}. Subsection~\ref{4.3} contains the construction of the transformation $G$ and provides the properties of $G$ needed to carry out steps~\eqref{sk1} and~\eqref{sk2}. The material presented in subsection~\ref{4.3} is essentially known from~\cite{LS15b}. The proof of Theorem~\ref{Thm1} is carried out in Subsection~\ref{4.4}. Subsection~\ref{4.5} contains the proof of Theorem~\ref{Thm2}.

Throughout the following we make use of the fact that the functions $\mu$ and $\sigma$ satisfy a linear growth condition, i.e. there exists $K\in(0, \infty)$ such that for all $x\in\R$, 
\begin{equation}\label{LG}
|\mu(x)|+|\sigma(x)|\leq K\cdot (1+|x|).
\end{equation}
This property is an immediate consequence of the assumptions (A1) and (A2).

\subsection{$L_p$-estimates of the solution  and the time-continuous Euler-Maruyama scheme}\label{4.1}
We have the following $L_p$-estimates for $X$, which follow from the linear growth property~\eqref{LG} of $\mu$ and $\sigma$ by using standard arguments as in~\cite[Sec.2.4]{Mao08}. 
\begin{lemma}\label{solprop}
Let $p\in[1, \infty)$. Then there exists  $c\in(0, \infty)$ such that for all $\delta\in[0,1]$ and all $t\in[0, 1-\delta]$,
\[
\bigl(\EE\bigl[\sup_{s\in[t, t+\delta]} |X(s)-X(t)|^p\bigr]\bigr)^{1/p}\leq c\cdot \sqrt{\delta}.
\]
In particular,
\[
\EE\bigl[\|X\|_\infty^p\bigr] < \infty.
\]
\end{lemma}

For technical reasons we have to provide $L_p$-estimates and some further properties
of the time-continuous Euler-Maruyama scheme for the SDE~\eqref{sde0} dependent on the initial value $x_0$. To be formally precise, for every $x\in\R$ we let 
$X^x$ denote the unique strong solution of the SDE
\begin{equation}\label{sde00}
\begin{aligned}
dX^x_t & = \mu(X^x_t) \, dt + \sigma(X^x_t) \, dW_t, \quad t\in [0,1],\\
X^x_0 & = x,
\end{aligned}
\end{equation}
and for all $x\in\R$ and $n\in\N$ we use $\eul_{n}^x=(\eul_{n,t}^x)_{t\in[0,1]}$ to denote the time-continuous Euler-Maruyama scheme with step-size $1/n$ associated to the SDE \eqref{sde00}, i.e. $\eul_{n,0}^x=x$ and
\[
\eul_{n,t}^x=\eul_{n,\utn}^x+\mu(\eul_{n,\utn}^x)\cdot (t-\utn)+\sigma(\eul_{n,\utn}^x)\cdot (W_t-W_{\utn})
\]
for $t\in [0,1]$. 
In particular, $X=X^{x_0}$ and $\widehat X_n = \eul_{n}^{x_0}$ for every $n\in\N$. Furthermore,   the integral representation
\begin{equation}\label{intrep}
\eul_{n,t}^x=x+\int_0^t \mu(\eul_{n,\usn}^x)\, ds+\int_0^t\sigma(\eul_{n,\usn}^x)\,dW_s
\end{equation}
holds for every $n\in\N$ and $t\in[0,1]$.

We have the following uniform $L_p$-estimates for $\eul_{n}^x$, $n\in\N$, which follow from \eqref{intrep} and the linear growth property~\eqref{LG} of $\mu$ and $\sigma$ by using standard arguments.

\begin{lemma}\label{eulprop}
Let $p\in[1, \infty)$. Then there exists  $c\in(0, \infty)$ such that for all $x\in\R$, all $n\in\N$, all $\delta\in[0,1]$ and all $t\in[0, 1-\delta]$,
\[
\bigl(\EE\bigl[\sup_{s\in[t, t+\delta]} |\eul_{n,s}^x-\eul_{n,t}^x|^p\bigr]\bigr)^{1/p}\leq c\cdot (1+|x|)\cdot \sqrt{\delta}.
\]
In particular,
\[
\sup_{n\in\N} \bigl(\EE\bigl[\|\eul_{n}^x\|_\infty^p\bigr]\bigr)^{1/p}\leq c\cdot (1+|x|).
\]
\end{lemma}

\subsection{A Markov property and occupation time estimates for the time-continuous Euler-Maruyama scheme}\label{4.2}

The following lemma provides a Markov property of the time-continuous Euler-Maruyama 
scheme
$\eul^x_n$ relative to the gridpoints $1/n,2/n,\ldots,1$.

\begin{lemma}\label{markov}
For all $x\in\R$, all $n\in\N$, all $j\in\{0, \ldots, n-1\}$ and $\PP ^{\eul_{n,j/n}^x} $-almost all $y\in\R$ we have
\[
\PP ^{(\eul_{n,t}^x)_{t\in [j/n, 1]}|\mathcal F_{j/n}}=\PP ^{(\eul_{n,t}^x)_{t\in [j/n, 1]}|\eul_{n,j/n}^x}
\]
as well as
\[
\PP ^{(\eul_{n,t}^x)_{t\in [j/n, 1]}|\eul_{n,j/n}^x=y}=\PP ^{(\eul^y_{n,t})_{t\in [0,1-j/n]}}.
\]
\end{lemma}
\begin{proof}
The lemma is an immediate consequence of the fact that, by definition of $\eul_n^x$, for every $\ell\in \{1,\ldots,n\}$  there exists a mapping $\psi\colon \R\times C([0,\ell/n]) \to C([0,\ell/n]) $ such that for all $x\in\R$ and all $i\in\{0,1,\ldots,n-\ell\}$,
\[
(\eul^x_{n,t+i/n})_{t\in[0, \ell/n]} = \psi\bigl(\eul^x_{n, i/n},(W_{t+i/n}-W_{i/n})_{t\in[0, \ell/n]}\bigr).\qedhere
\]
\end{proof}

Next, we provide an estimate for the expected occupation time  of a neighborhood of   a non-zero of $\sigma$ by the time-continuous Euler-Maruyama scheme $\eul_{n}^x$.

\begin{lemma}\label{occup}
Let $\xi\in\R$ satisfy $\sigma(\xi)\not=0$. Then there exists $ c\in (0, \infty)$ such that for   all $x\in\R$, all $n\in\N$ and all $\eps\in (0,\infty)$,
\begin{equation}
\int_0^1  \PP(\{|\eul_{n,t}^x-\xi|\leq \varepsilon\})\,dt\leq c\cdot (1+x^2)\cdot\Bigl(\varepsilon+\frac{1}{\sqrt n}\Bigr). 
\end{equation}
\end{lemma}

\begin{proof}
Let  $x\in\R$ and $n\in\N$. 
By~\eqref{intrep},
\eqref{LG}
and Lemma~\ref{eulprop} we see that $\eul_n^x$ is a continuous semi-martingale with quadratic variation
\begin{equation}\label{qv}
\langle \eul_{n}^x\rangle_t
=x^2+\int_0^t \sigma^2\bigl(\eul_{n,\usn}^x\bigr)\, ds,\quad t\in[0,1].
\end{equation}
For $a\in\R$ let $L^a(\eul_n^x) = (L^a_t(\eul_n^x))_{t\in[0,1]}$ denote the local time 
of $\eul_{n}^x$ at the point $a$.
Thus, 
for all $a\in\R$ and
 all $t\in[0,1]$, 
\begin{align*}
|\eul_{n,t}^x-a| & = |x-a| + \int_0^t \sgn(\eul_{n,s}^x-a)\cdot \mu (\eul_{n,s}^x)\, ds + \int_0^t \sgn(\eul_{n,s}^x-a)\cdot \sigma (\eul_{n,s}^x)\, dW_s + L^a_t(\eul_n^x),
\end{align*}
where $\sgn(z) = 1_{(0,\infty)}(z) - 1_{(-\infty,0]}(z)$ for $z\in\R$,
see, e.g.~\cite[Chap. VI]{RevuzYor2005}.
Hence, 
for all $a\in\R$ and
 all $t\in[0,1]$,
\begin{align*}
L^a_t(\eul_n^x) & \le |\eul^x_{n,t}-x| + \int_0^t |\mu (\eul_{n,s}^x)|\, ds + \Bigl|\int_0^t \sgn(\eul_{n,s}^x-a)\cdot \sigma (\eul_{n,s}^x)\, dW_s\Bigr|.
\end{align*}

Using the H\"older inequality, the Burkholder-Davis-Gundy inequality, \eqref{LG} and the second estimate in Lemma~\ref{eulprop} we conclude that 
there exists $c\in (0,\infty)$ such that
  for all $x\in\R$, all $n\in\N$,
all $a\in\R$ 
   and all $t\in[0,1]$,
\begin{equation}\label{local1}
\EE\bigl[L^a_t(\eul_n^x)\bigr]  \le  c\cdot (1+|x|).
\end{equation}
Let $\varepsilon \in (0,\infty)$.
Using~\eqref{qv} and~\eqref{local1} we obtain by the occupation 
time
formula that
there exists $c\in (0,\infty)$ such that  
for all $x\in\R$, all $n\in\N$ and all $\eps\in (0,\infty)$,
\begin{equation}\label{local2}
  \EE\Bigl[\int_0^1 1_{[\xi-\eps,\xi+\eps]}(\eul^x_{n,t})\cdot \sigma^2(\eul_{n,\utn}^x)\, dt\Bigr]
 = \int_{\R}1_{[\xi-\eps,\xi+\eps]}(a)\, \EE\bigl[L^a_t(\eul_n^x)\bigr]\, da 
 \le c\cdot (1+|x|)\cdot \eps. 
\end{equation}

Using the Lipschitz continuity of $\sigma$ as well as~\eqref{LG} and Lemma~\ref{eulprop} 
we obtain that 
there exist  $c_1,c_2\in (0,\infty)$ such that 
for all $x\in\R$ and all $n\in\N$,
\begin{equation}\label{local3} 
\begin{aligned}
  \EE\Bigl[\int_0^1\bigl|\sigma^2(\eul^x_{n,t})-\sigma^2(\eul^x_{n,\utn})\bigr|\, dt \Bigr]
 & \le c_1 \cdot \int_0^1\EE\bigl[|\eul^x_{n,t}-\eul^x_{n,\utn}|\cdot (1 + \|\eul^x_{n}\|_\infty)\bigr]\, dt \\ & 
  \le c_2 \cdot (1+x^2)\cdot \frac{1}{\sqrt n}.
\end{aligned}
\end{equation}

Since $\sigma$ is continuous and $\sigma(\xi)\neq 0$ there exist  $\kappa,\eps_0\in(0,\infty)$ such that 
\[
\inf_{|z-\xi|< \eps_0}\sigma^2(z) \ge \kappa.
\]
Observing~\eqref{local2} and~\eqref{local3} we conclude that there exists $c\in (0,\infty)$ such that for all $x\in\R$, all $n\in\N$ and all $\eps \in (0,\eps_0]$,
\begin{align*}
 \int_0^1  \PP(\{|\eul_{n,t}^x-\xi|\leq \varepsilon\})\,dt  &  = \frac{1}{\kappa}\cdot \EE\Bigl[\int_0^1 \kappa\cdot 1_{[\xi-\eps,\xi+\eps]}(\eul^x_{n,t})\, dt\Bigr] \\ &  \le \frac{1}{\kappa}\cdot \EE\Bigl[\int_0^1 1_{[\xi-\eps,\xi+\eps]}(\eul^x_{n,t})\cdot  \sigma^2(\eul^x_{n,t})\, dt\Bigr] \\
& \le  \frac{1}{\kappa}\cdot \EE\Bigl[\int_0^1\bigl( 1_{[\xi-\eps,\xi+\eps]}(\eul^x_{n,t})\cdot  \sigma^2(\eul^x_{n,\utn}) + \bigl|\sigma^2(\eul^x_{n,t})-\sigma^2(\eul^x_{n,\utn})\bigr|\bigr)\, dt\Bigr]\\
&  \le  \frac{c}{\kappa}\cdot (1+|x|+x^2)\cdot \Bigl( \eps + \frac{1}{\sqrt n}\Bigr),
\end{align*}  
which completes the proof of the lemma.
\end{proof}

The following result shows how to  transfer  the condition of a sign change of $\eul_n -\xi$ at time $t$ relative to its sign at the grid point $\utn$ to a condition on the distance of $\eul_n$ and $\xi$ at the time $\utn-(t-\utn)$. 

\begin{lemma}\label{central}
Let $\xi\in\R$. Then there exists $c\in(0, \infty)$ such that for all $n\in\N$, all $0\le s\le t \le 1$ with $\utn-s \ge 1/n$ and all $A\in \F_s$,
\begin{equation}\label{central1}
\begin{aligned}
& \PP\bigl(A \cap \{(\widehat X_{n,t}-\xi)\cdot( \widehat X_{n,\utn}-\xi)\le 0\}\bigr)\\
& \qquad\qquad \le \frac{c}{n}\cdot \PP(A) + c\cdot \int_{\R} \PP\bigl(A \cap \bigl\{|\widehat X_{n, \utn-(t-\utn)}-\xi| \le \tfrac{c}{\sqrt{n}} (1+|z|)\bigr\}\bigr)\cdot e^{-\frac{z^2}{2}}\, dz.
\end{aligned}
\end{equation}
\end{lemma}

\begin{proof}
Choose $K\in (0,\infty)$ according to~\eqref{LG} and choose $n_0\in\N\setminus\{1\}$ 
such that for all $n \geq n_0$,
\[
12K\cdot (1+|\xi|) \cdot\frac{1+\sqrt{2\ln(n)}}{\sqrt{n}} \le \frac{1}{2}.
\]
Without loss of generality we may assume that $n\ge n_0$.
Let $ 0\le s\le t \le 1$ with $\utn-s \ge 1/n$ and let $A\in \F_s$.
If $t=\utn$ then for all $c\in (0,\infty)$ and all $z\in\R$ we have
\[
 \{(\widehat X_{n,t}-\xi)\cdot (\widehat X_{n,\utn}-\xi)\le 0\} =  \{\widehat X_{n,\utn}-\xi = 0\} \subset  \bigl\{|\widehat X_{n,\utn-(t-\utn)}-\xi| \le \tfrac{c}{\sqrt{n}} (1+|z|)\bigr\},
\]
which implies that in this case~\eqref{central1} holds for all $c\ge 1/\sqrt{2\pi}$.

Now assume that $t > \utn$ and put 
\[
Z_1 = \frac{W_t-W_{\utn}}{\sqrt{t-\utn}},\quad Z_2 = \frac{W_{\utn}-W_{\utn-(t-\utn)}}{\sqrt{t-\utn}},\quad Z_3=\frac{W_{\utn-(t-\utn)}-W_{\utn-1/n} }{\sqrt{1/n-(t-\utn)}}.
\]
Below we show that
\begin{equation}\label{central2}
\begin{aligned}
& \bigl\{(\widehat X_{n,t}-\xi)\cdot (\widehat X_{n,\utn}-\xi)\le 0\bigr\}\cap \bigl\{\max_{i\in\{1,2,3\}}|Z_i| \le \sqrt{2\ln(n)}\bigr\}  \\
& \qquad\qquad \subset \bigl\{|\widehat X_{n, \utn-(t-\utn)}-\xi| \le 12K\cdot (1+|\xi|)\cdot (1+|Z_1|+|Z_2|)/\sqrt{n}\bigr\}. 
\end{aligned}
\end{equation}
Note that $Z_1,Z_2,Z_3$ are independent and identically  distributed standard normal random variables. Moreover, $(Z_1,Z_2,Z_3)$ is independent of $\F_s$ since $s \le \utn-1/n$, $(Z_1,Z_2)$ is independent of $\F_{\utn-(t-\utn)}$ and $\widehat X_{n, \utn-(t-\utn)}$ is $\F_{\utn-(t-\utn)}$-measurable. Using the latter facts jointly with~\eqref{central2} and a standard estimate of standard normal tail probabilities we  obtain that
\begin{align*}
& \PP\bigl(A \cap \{(\widehat X_{n,t}-\xi)\cdot( \widehat X_{n, \utn}-\xi)\le 0\}\bigr)\\
& \quad \le \PP\bigl(A\cap \{|\widehat X_{n, \utn-(t-\utn)}-\xi| \le 12K\cdot (1+|\xi|)\cdot (1+|Z_1|+|Z_2|)/\sqrt{n}\}\bigr) \\
&\qquad \qquad \qquad   + \PP\bigl(A\cap \bigl\{\max_{i\in\{1,2,3\}}|Z_i| > \sqrt{2\ln(n)}\bigr\}\bigr)\\
&  \quad  \le \frac{2}{\pi}\int_{[0,\infty)^2} \PP\bigl(A\cap\bigl\{|\widehat X_{n,\utn-(t-\utn)}-\xi| \le 12K\cdot (1+|\xi|)\cdot \tfrac{ 1+z_1+z_2}{\sqrt{n}}\bigr\}\bigr)\cdot e^{-\frac{z_1^2+z_2^2}{2}}\, d(z_1,z_2)\\
&\qquad \qquad \qquad   +6\PP(A)\cdot \PP\bigl(\{Z_1 > \sqrt{2\ln(n)}\}\bigr)\\
&  \quad  \le \frac{2}{\pi}\int_{\R^2} \PP\Bigl(A\cap\Bigl\{|\widehat X_{n,\utn-(t-\utn)}-\xi| \le 12\sqrt{2}K\cdot (1+|\xi|)\cdot \tfrac{ 1+|\tfrac{z_1+z_2}{\sqrt{2}}|}{\sqrt{n}}\Bigr\}\Bigr)\cdot e^{-\frac{z_1^2+z_2^2}{2}}\, d(z_1,z_2)\\
&\qquad \qquad \qquad   +\frac{6\PP(A)}{\sqrt{2\pi\cdot 2\ln(n)}\cdot n}\\
&  \quad  = \frac{4}{\sqrt{2\pi}}\int_{\R} \PP\bigl(A\cap\bigl\{|\widehat X_{n,\utn-(t-\utn)}-\xi| \le 12\sqrt{2}K\cdot (1+|\xi|)\cdot \tfrac{ 1+|z|}{\sqrt{n}}\bigr\}\bigr)\cdot e^{-\frac{z^2}{2}}\, dz + \frac{3\PP(A)}{\sqrt{\pi \ln(n)}\cdot n},
\end{align*}
which yields~\eqref{central1}.

It remains to prove the inclusion~\eqref{central2}. To this end let $\omega\in\Omega$ and assume that
\begin{equation}\label{central3}
(\widehat X_{n,t}(\omega)-\xi)\cdot (\widehat X_{n,\utn}(\omega)-\xi)\le 0\text{\quad and \quad} \max_{i\in\{1,2,3\}}|Z_i(\omega)| \le \sqrt{2\ln(n)}.
\end{equation}
Using ~\eqref{LG}  and the fact that for all $a,b\in\R$,
\begin{equation}\label{JJJ}
1+|a|\leq (1+|a-b|)\cdot (1+|b|),
\end{equation}
we obtain
\begin{equation}\label{LLL}
\begin{aligned}
|\eul_{n,\utn}(\omega) -\xi| & \le |(\eul_{n,\utn}(\omega) -\xi) - (\eul_{n,t}(\omega) -\xi)| \\
&  = |\mu(\eul_{n,\utn}(\omega))\cdot (t-\utn) + \sigma(\eul_{n,\utn}(\omega))\cdot \sqrt{t-\utn}\cdot Z_1(\omega)|\\
& \le K\cdot(1+|\eul_{n,\utn}(\omega)|)\cdot\Bigl(\frac{1}{n} +\frac{1}{\sqrt{n}}\cdot |Z_1(\omega)|\Bigr)\\
& \le (1+|\eul_{n,\utn}(\omega)-\xi|)\cdot K\cdot(1+|\xi|)\cdot \frac{1}{\sqrt{n}}\cdot (1+|Z_1(\omega)|).
\end{aligned}
\end{equation}
Since $n \geq n_0$ 
we have  
\[
K\cdot(1+|\xi|)\cdot \frac{1}{\sqrt{n}}\cdot (1+|Z_1(\omega)|)\le  K\cdot (1+|\xi|)\cdot \frac{1+ \sqrt{2\ln(n)}}{\sqrt{n}} \leq \frac{1}{2},
\]
and therefore,
\begin{equation}\label{central4}
|\eul_{n,\utn}(\omega)-\xi|\leq \frac{K\cdot  (1+|\xi|)\cdot\tfrac{1 }{\sqrt n}\cdot (1+|Z_1(\omega)|)}{1-K\cdot (1+|\xi|)\cdot\tfrac{1 }{\sqrt n}\cdot (1+|Z_1(\omega)|)}\leq 2K\cdot  (1+|\xi|)\cdot\frac{1}{\sqrt n}\cdot (1+|Z_1(\omega)|).
\end{equation}
Similarly to \eqref{LLL}, we obtain by \eqref{LG} and \eqref{JJJ} that
\begin{equation}\label{central5}
|\eul_{n,\utn}(\omega)-\eul_{n,\utn-(t-\utn)}(\omega)| \le (1+|\eul_{n,\utn-1/n}(\omega)-\xi|)\cdot K\cdot (1+|\xi|)\cdot \frac{1}{\sqrt{n}}\cdot (1+|Z_2(\omega)|)
\end{equation} 
and 
\begin{equation}\label{central6}
|\eul_{n,\utn-(t-\utn)}(\omega)-\eul_{n,\utn-1/n}(\omega)| \le (1+|\eul_{n,\utn-1/n}(\omega)-\xi|)\cdot K\cdot (1+|\xi|)\cdot \frac{1}{\sqrt{n}}\cdot (1+|Z_3(\omega)|).
\end{equation} 
Since 
$n\geq n_0$
we have $ K\cdot (1+|\xi|)\cdot \tfrac{1}{\sqrt{n}}\cdot (1+|Z_3(\omega)|)\leq 1/2$, 
and therefore we conclude from~\eqref{central6} that
\begin{equation}\label{central7}
\begin{aligned}
 1 +|\eul_{n,\utn-(t-\utn)}(\omega) -\xi|  &\ge 1 + |\eul_{n,\utn-1/n}(\omega)-\xi| -|\eul_{n,\utn-(t-\utn)}(\omega)-\eul_{n,\utn-1/n}(\omega)|\\
&\ge (1 + |\eul_{n,\utn-1/n}(\omega)-\xi|)/2.
\end{aligned}
\end{equation} 
Using~\eqref{central4}, \eqref{central5} and~\eqref{central7} we obtain
\begin{equation}\label{central8}
\begin{aligned}
& |\eul_{n,\utn-(t-\utn)}(\omega) -\xi|\\
 &\qquad \qquad  \le |\eul_{n,\utn}(\omega)-\eul_{n,\utn-(t-\utn)}(\omega)| +|\eul_{n,\utn}(\omega)-\xi|\\
& \qquad \qquad \le (1+|\eul_{n,\utn-1/n}(\omega)-\xi|)\cdot 3K\cdot(1+|\xi|)\cdot \tfrac{1}{\sqrt{n}}\cdot (1+|Z_1(\omega)|+|Z_2(\omega)|)\\
& \qquad \qquad \le (1+|\eul_{n,\utn-(t-\utn)}(\omega)-\xi|)\cdot 6K\cdot(1+|\xi|)\cdot \tfrac{1}{\sqrt{n}}\cdot (1+|Z_1(\omega)|+|Z_2(\omega)|).
\end{aligned}
\end{equation} 
Since $n\geq n_0$ 
we have $6K\cdot(1+|\xi|)\cdot \tfrac{1}{\sqrt{n}}\cdot (1+|Z_1(\omega)|+|Z_2(\omega)|)\le 1/2$,
which jointly with~\eqref{central8} yields
\[
|\eul_{n,\utn-(t-\utn)}(\omega) -\xi| \le 12K\cdot (1+|\xi|)\cdot \tfrac{1}{\sqrt{n}}\cdot (1+|Z_1(\omega)|+|Z_2(\omega)|).
\]
This finishes the proof of~\eqref{central2}.
\end{proof}

Using Lemmas~\ref{markov},~\ref{occup} and~\ref{central} we can now establish the following two estimates on the probability of sign changes of $\eul_{n}-\xi$ relative to its sign at the gridpoints $0,1/n,\dots,1$. 

\begin{lemma}\label{key}
Let $\xi\in\R$ satisfy $\sigma(\xi)\neq 0$ and let
\[
A_{n,t} =\{(\eul_{n,t}-\xi)\cdot(\eul_{n, \utn}-\xi)\leq 0\}
\]
for all $n\in\N$ and $t\in[0,1]$.  
Then the following 
two
statements hold.
\begin{itemize}
\item[(i)] There exists $c\in(0, \infty)$ such that for all $n\in\N$, all $s\in [0,1)$ and all $A\in \F_s$,
\[
 \int_{s}^1\PP(A \cap A_{n,t})\, dt \le \frac{c}{\sqrt{n}} \cdot \bigl( \PP(A) + \EE\bigl[\ind_A \cdot (\eul_{n,\usn+1/n}-\xi)^2\bigr]\bigr).
\]
\item[(ii)] There exists $c\in(0, \infty)$ such that for all $n\in\N$, all $s\in[0,1)$ and all $A\in \F_{s}$,
\[
\int_s^1\EE\bigl[\ind_{A\cap A_{n,t}} \cdot (\eul_{n,\utn+1/n}-\xi)^2\bigr]\, dt \le  \frac{c}{n} \cdot \bigl(\PP(A) + \EE\bigl[\ind_A \cdot (\eul_{n,\usn+1/n}-\xi)^2\bigr]\bigr).
\]
\end{itemize}
\end{lemma} 

\begin{proof}
Let $n\in\N$, $s\in [0,1)$ and $A\in\F_s$. 
In the following we use $c_1,c_2,\dots \in (0,\infty)$ to denote unspecified
positive constants, which neither depend on $n$ nor on $s$ nor on $A$.

We first prove part (i) of the lemma.
Clearly we may assume that $s<1-1/n$. Then $\usn\le 1-2/n$ and we have
\begin{equation}\label{key00}
 \int_{s}^1\PP(A \cap A_{n,t})\, dt \le \frac{2}{n}\cdot \PP(A) +  \int_{\usn + 2/n}^1\PP(A \cap A_{n,t})\, dt.
\end{equation}
 If $t\in [\usn+2/n,1]$ then $\utn \ge \usn+2/n$, which implies $\utn-1/n \ge \usn+1/n \ge s$. We may thus apply Lemma~\ref{central} to conclude
 that 
there exists $c_1\in(0, \infty)$ such that
\begin{equation}\label{key01}
\begin{aligned}
& \int_{s}^1 \PP(A\cap A_{n,t})\, dt\\
& \qquad \le \frac{c_1}{n}\cdot \PP(A) + c_1\cdot \int_{\R} \int_{\usn+2/n}^1 \PP\bigl(A \cap \bigl\{|\widehat X_{n,\utn-(t-\utn)}-\xi| \le \tfrac{c_1}{\sqrt{n}} (1+|z|)\bigr\}\bigr)\cdot e^{-\frac{z^2}{2}}\, dz
\, dt \\
& \qquad = \frac{c_1}{n}\cdot \PP(A) + c_1\cdot \int_{\R} \int_{\usn+1/n}^{1-1/n} \PP\bigl(A \cap \bigl\{|\widehat X_{n,t}-\xi| \le \tfrac{c_1}{\sqrt{n}} (1+|z|)\bigr\}\bigr)\cdot e^{-\frac{z^2}{2}}\, dz\, dt.
\end{aligned}
\end{equation}
 By the fact that $A\in \F_{\usn+1/n}$ and by the first part of Lemma~\ref{markov} we obtain that for all $z\in\R$,
\begin{equation}\label{key02}
\begin{aligned}
& \int_{\usn+1/n}^{1-1/n} \PP\bigl(A \cap \bigl\{|\widehat X_{n,t}-\xi| \le \tfrac{c_1}{\sqrt{n}} (1+|z|)\bigr\}\bigr)\, dt \\
& \qquad\qquad \qquad = \EE\Bigl[\ind_{A}\cdot \EE\Bigl[\int_{\usn+1/n}^{1-1/n} \ind_{\{|\widehat X_{n,t}-\xi| \le \tfrac{c_1}{\sqrt{n}} (1+|z|)\}}\, dt\Bigl|\eul_{n,\usn+1/n}\Bigr]\Bigr].
\end{aligned}
\end{equation}
Moreover, by the second part of Lemma~\ref{markov} and by Lemma~\ref{occup} we obtain that there exists $c_2\in(0, \infty)$ such that  for all $z,x\in\R$, 
\begin{equation}\label{key03}
\begin{aligned}
& \EE\Bigl[\int_{\usn+1/n}^{1-1/n} \ind_{\{|\widehat X_{n,t}-\xi| \le \tfrac{c_1}{\sqrt{n}} (1+|z|)\}}\, dt\Bigl|\eul_{n,\usn+1/n}=x\Bigr]\\
 & \qquad= \EE\Bigl[\int_{0}^{1-2/n-\usn} \ind_{\{|\widehat X^x_{n,t}-\xi| \le \tfrac{c_1}{\sqrt{n}} (1+|z|)\}}\, dt\Bigr] \le c_2\cdot (1+x^2)\cdot \Bigl( \frac{c_1}{\sqrt{n}} \cdot(1+|z|) + \frac{1}{\sqrt{n}}\Bigr).
\end{aligned}
\end{equation}
Combining~\eqref{key02} and~\eqref{key03} 
and using the fact that for all $a,b\in\R$,
\[
1+a^2\leq 2\,(1+(a-b)^2)\cdot (1+b^2),
\]
we conclude that for  all $z\in\R$,
\begin{equation}\label{key04}
\begin{aligned}
& \int_{\usn+1/n}^{1-1/n} \PP\bigl(A \cap \bigl\{|\widehat X_{n,t}-\xi| \le \tfrac{c_1}{\sqrt{n}} (1+|z|)\bigr\}\bigr)\, dt \\ & \qquad\qquad\qquad\qquad  \le \tfrac{c_2(c_1+1)}{\sqrt{n}}\cdot (1+|z|)\cdot \EE\bigl[\ind_{A}\cdot (1 + \eul_{n,\usn+1/n}^2)\bigr]\\
 & \qquad\qquad\qquad\qquad  \le \tfrac{2c_2(c_1+1)}{\sqrt{n}}\cdot (1+\xi^2)\cdot (1+|z|)\cdot \bigl(\PP(A) + \EE\bigl[\ind_{A}\cdot (\eul_{n,\usn+1/n}-\xi)^2\bigr]\bigr).
\end{aligned}
\end{equation}
Inserting~\eqref{key04} into~\eqref{key01} and observing that $\int_\R (1+|z|)\cdot e^{-z^2/2}\,dz < \infty$ completes the proof of part (i) of the lemma. 

We next prove part (ii).
Clearly,
\begin{align*}
& \int_s^{1} \EE\bigl[\ind_{A\cap A_{n,t}} \cdot (\eul_{n,\utn+1/n}-\xi)^2\bigr]\, dt \\
& \qquad =\int_s^{\usn+1/n} \EE\bigl[\ind_{A\cap A_{n,t}} \cdot (\eul_{n,\utn+1/n}-\xi)^2\bigr]\, dt+ \int_{\usn+1/n}^1 \EE\bigl[\ind_{A\cap A_{n,t}} \cdot (\eul_{n,\utn+1/n}-\xi)^2\bigr]\, dt. 
\end{align*}

If $t\in[s, \usn+1/n)$ then $\underline t_n=\underline s_n$ and therefore
\begin{equation}\label{key003}
\begin{aligned}
 \int_s^{\usn+1/n} \EE\bigl[\ind_{A\cap A_{n,t}} \cdot (\eul_{n,\utn+1/n}-\xi)^2\bigr]\, dt&=\int_s^{\usn+1/n}\EE\bigl[\ind_{A\cap A_{n,t}} \cdot (\eul_{n,\usn+1/n}-\xi)^2\bigr]\, dt\\
& \le  \int_s^{\usn+1/n}  \EE\bigl[\ind_A \cdot (\eul_{n,\usn+1/n}-\xi)^2\bigr]\, dt 
\\ &  \le \frac{1}{n}\cdot  \EE\bigl[\ind_A\cdot (\eul_{n,\usn+1/n}-\xi)^2\bigr].
\end{aligned}
\end{equation}

Next, let $t\in [\usn+1/n, 1]$.  Clearly, we have on $A_t$,
\begin{align*}
|\eul_{n,\utn+1/n}-\xi| &\le |\eul_{n,\utn+1/n}-\eul_{n,t}| + |\eul_{n,t}-\xi| \le |\eul_{n,\utn+1/n}-\eul_{n,t}| + |\eul_{n,t}-\eul_{n,\utn}|.
\end{align*}
Hence, by Lemma~\ref{markov}(i),
\begin{equation}\label{key05}
\begin{aligned}
& \EE\bigl[\ind_{A\cap A_{n,t}}\cdot (\eul_{n,\utn+1/n} - \xi)^2\bigr] \\ &  \qquad\quad  \le \EE \bigl[\ind_A \cdot (|\eul_{n,\utn+1/n}-\eul_{n,t}| + |\eul_{n,t}-\eul_{n,\utn}|)^2 \bigr]\\ & \qquad\quad = \EE\bigl[ \ind_A \cdot \EE\bigl[(|\eul_{n,\utn+1/n}-\eul_{n,t}| + |\eul_{n,t}-\eul_{n,\utn}|)^2\bigl| \eul_{n,\usn+1/n}\bigr]\bigr].
\end{aligned}
\end{equation}

If $t\ge \usn+1/n$ then $\utn \ge \usn+1/n$. Hence, by Lemma~\ref{markov}(ii) and Lemma~\ref{eulprop} we obtain that there exist $c_1, c_2\in(0, \infty)$ such that  for all  $t\in [\usn+1/n, 1]$ and all $x\in\R$,
\begin{equation}\label{key06}
\begin{aligned}
&  \EE\bigl[(|\eul_{n,\utn+1/n}-\eul_{n,t}| + |\eul_{n,t}-\eul_{n,\utn}|)^2\bigl| \eul_{n,\usn+1/n} = x\bigr]\\
& \qquad \quad = \EE[(|\eul^x_{n,\utn-\usn}-\eul^x_{n,t-\usn-1/n}| + |\eul^x_{n,t-\usn-1/n}-\eul^x_{\utn-\usn-1/n}|)^2\bigr] 
\\& \qquad\quad \le  c_1\cdot (1+x^2)\cdot 1/n
\le  c_2\cdot (1+(x-\xi)^2)\cdot 1/n.
\end{aligned}
\end{equation}
It follows from~\eqref{key05} and~\eqref{key06} that
\begin{equation}\label{key07}
\begin{aligned}
& \int_{\usn+1/n}^1\EE\bigl[\ind_{A\cap A_{n,t}}\cdot (\eul_{n,\utn+1/n} - \xi)^2\bigr]\, dt \\
& \qquad\qquad  \le \frac{c_2}{n}
\cdot \int_{\usn+1/n}^1 \EE\bigl[\ind_A\cdot (1+(\eul_{n,\usn+1/n}-\xi)^2)\bigr]\, dt\\
& \qquad\qquad \le \frac{c_2}{n}
\cdot \bigl(\PP(A) + \EE\bigl[\ind_A\cdot (\eul_{n,\usn+1/n}-\xi)^2)\bigr]\bigr).
\end{aligned}
\end{equation}
Combining~\eqref{key003} with~\eqref{key07} completes the proof of part (ii) of the lemma. 
\end{proof}

We are ready to establish the main result in this section, which provides a $p$-th mean estimate of the 
Lebesgue measure of the set
of times $t$ of a sign change of $\eul_{n,t}-\xi$ relative to the sign of $\eul_{n,\utn}-\xi$.

\begin{prop}\label{prop1} 
Let $\xi\in\R$ satisfy $\sigma(\xi)\not=0$ and let
 $p\in [1,\infty)$.
Then there exists  $c\in(0, \infty)$ such that for all $n\in\N$, 
\begin{equation}\label{l33}
\EE\Bigl[\Bigl|\int_0^1  \ind_{\{(\eul_{n,t}-\xi)\cdot(\eul_{n, \utn}-\xi)\leq 0\}}\,dt\Bigr|^p\Bigr]^{1/p}\leq \frac{c}{\sqrt{n}}. 
\end{equation}
\end{prop}

\begin{proof}
Clearly, it suffices to consider only the case $p\in\N$.
For $n\in\N$ and $t\in [0,1]$
put $A_{n,t} =\{(\eul_{n,t}-\xi)\cdot(\eul_{n, \utn}-\xi)\leq 0\}$   as in Lemma~\ref{key}, and for $n,p\in\N$ let
\[
a_{n,p} = \EE\Bigl[\Bigl(\int_0^1 \ind_{A_{n,t}}\, dt\Bigr)^p\Bigr].
\]
 We prove by induction on $p$ that for every $p\in\N$ there exists $c\in (0,\infty)$ such that for all $n\in\N$,
\begin{equation}\label{prop01}
a_{n,p} \le c\cdot n^{-p/2}.
\end{equation}

First assume that $p=1$. Using Lemma~\ref{key}(i) with $s=0$ and $A=\Omega$ we 
obtain that there exists $c\in (0,\infty)$ such that for all $n\in\N$,
\[
a_{n,1} = \int_0^1 \PP(A_{n,t})\, dt  \le \frac{c}{\sqrt{n}}\cdot (1 + \EE[(\eul_{n,1/n}-\xi)^2 ]) \le \frac{c}{\sqrt{n}}\cdot (1+2\xi^2+2\sup_{j\in\N}\EE[\|\eul_j\|_\infty^2]).
\]
Observing Lemma~\ref{eulprop} we thus see that~\eqref{prop01} holds for $p=1$. 

Next, let $q\in\N$ 
and assume that~\eqref{prop01} holds for all $p\in\{1, \ldots, q\}$.
Clearly,
\begin{align*}
a_{n,q+1}  &= (q+1)!\cdot \int_0^1\int_{t_1}^1\ldots \int_{t_q}^1 \PP(A_{n,t_1}\cap A_{n,t_2}\cap \ldots \cap A_{n,t_{q+1}})\,dt_{q+1}\, \ldots \,dt_2\, dt_1. 
\end{align*}

First applying Lemma~\ref{key}(i) with $A= A_{n,t_1}\cap \ldots \cap A_{n,t_q}$ and $s= t_{q}$, then applying $(q-1)$-times Lemma~\ref{key}(ii) with $A= A_{n,t_1}\cap \ldots \cap A_{n,t_j}$ and $s= t_j$ for $j=q-1,\ldots, 1$, and finally applying Lemma~\ref{key}(ii) with $A=\Omega$ and $s=0$ we conclude that there exist constants $c_1, c_2, c_3\in(0, \infty)$ such that for all $n\in\N$,
\begin{align*}
a_{n,q+1} & \le  \frac{c_1}{\sqrt{n}}\cdot\Bigl( a_{n,q}+\int_0^1\ldots \int_{t_{q-1}}^1 \EE\bigl[\ind_{A_{n,t_1}\cap \ldots \cap A_{n,t_q}}\cdot (\eul_{n,\underline{t_q}_n+1/n}-\xi)^2 ])\,dt_q \, \ldots \, dt_1 \Bigr)\\
&  \le c_2\cdot\Bigl(\frac{a_{n,q}}{\sqrt{n}}  + \frac{a_{n,q-1}}{n^{3/2}}+\ldots + \frac{a_{n,1} }{n^{q-1/2}}+ \frac{1}{n^{q-1/2}}\cdot\int_0^1  \EE\bigl[\ind_{A_{n,t_1}}\cdot (\eul_{n,\underline{t_1}_n+1/n} - \xi)^2 \bigr] \, dt_1\Bigr)\\
&  \le c_2\cdot\Bigl(\frac{a_{n,q}}{\sqrt{n}}  + \frac{a_{n,q-1}}{n^{3/2}}+\ldots + \frac{a_{n,1} }{n^{q-1/2}}+ \frac{c_3}{n^{q+1/2}}\cdot \bigl(1 + 2\xi^2 + 2\sup_{j\in\N}\EE[\bigl\|\eul_j\|_\infty^2\bigr] \bigr)\Bigr).
\end{align*}
Employing Lemma~\ref{eulprop} and the induction hypothesis yields the validity of~\eqref{prop01} for $p=q+1$, which finishes the proof of the proposition.
\end{proof}

\subsection{The transformed equation}\label{4.3}

We turn to the construction and the properties of the mapping $G\colon\R\to\R$ that is used to switch from the SDE~\eqref{sde0} to an SDE with Lipschitz continuous coefficients.
The material presented in this subsection is essentially known from~\cite{LS15b}.

\begin{lemma}\label{transform1}
There exists a function $G\colon \R\to\R$ with the following properties.
\begin{itemize}
\item[(i)] $G$ is differentiable with 
\[
0<\inf_{x\in\R} G'(x)\leq \sup_{x\in\R} G'(x)<\infty.
\]
In particular, $G$ is Lipschitz continuous and has an inverse $G^{-1}\colon \R\to \R$ that is Lipschitz continuous as well.
\item[(ii)] The derivative $G'$ of $G$ is Lipschitz continuous hence absolutely continuous.  Moreover, $G'$ has a bounded Lebesgue-density $G''\colon \R\to \R$ that
is Lipschitz continuous on each of the intervals $(\xi_0,\xi_1 ),\dots,(\xi_k,\xi_{k+1})$ and such that
the functions 
\[
\widetilde \mu=(G'\cdot \mu+\tfrac{1}{2}G''\cdot\sigma^2)\circ G^{-1} \, \text{ and }\, \widetilde\sigma=(G'\cdot\sigma)\circ G^{-1}
\] 
are Lipschitz continuous.
\end{itemize}
\end{lemma}

\begin{proof}
We only provide a sketch of the proof.
If $k=0$ then $\mu$ and $\sigma$ are Lipschitz continuous and we can take $G(x) = x$ for all $x\in \R$. 

Now, assume that $k\in \N$. Since $\mu$ is Lipschitz continuous on each of the intervals $(\xi_0,\xi_1),\ldots$, $(\xi_k,\xi_{k+1})$ it is easy to see that the one-sided limits $\mu(\xi_i-)$ and $\mu(\xi_i+)$ exist for all $i\in\{1,\ldots,k\}$. 
For $i\in\{1, \ldots, k\}$ put
\[
\alpha_i=\frac{\mu(\xi_i-)-\mu(\xi_i+)}{2 \sigma^2(\xi_i)},
\]
let $\rho\in(0, \infty]$ be given by
\[
\rho=  \begin{cases}
\frac{1}{6 |\alpha_1|}, & \text{if }k=1, \\
\min\bigl(\bigl\{\frac{1}{6 |\alpha_i|}\colon i\in \{1, \ldots, k\}\bigr\} \cup \bigl\{ \frac{\xi_i-\xi_{i-1}}{2}\colon i\in \{2, \ldots, k\}\bigr\} \bigr),& \text{if }k\geq 2,
\end{cases}
\]
where we use the convention $1/0=\infty$, let $\nu\in (0, \rho)$, 
let $\phi\colon\R\to\R$ be given by
\[
\phi(x)=(1-x^2)^3\cdot \ind_{[-1, 1]}(x), 
\]
and define $G\colon\R\to\R$ by
\[
G(x)=x+\sum_{i=1}^k \alpha_i\cdot (x-\xi_i)\cdot |x-\xi_i|\cdot \phi \Bigl(\frac{x-\xi_i}{\nu}\Bigr).
\]

It is  straightforward to check that $G$ is differentiable with 
$\sup_{x\in\R} G'(x)<\infty$.
For the proof of $\inf_{x\in\R} G'(x)>0$ see Lemma 2.2 in \cite{LS15b}.

Put $\Theta=\{\xi_1, \ldots, \xi_k\}$. It is  straightforward to check that $G'$ is Lipschitz continuous and continuously differentiable on $\R\setminus\Theta$, $(G_{|\R\setminus\Theta})''$ is bounded, Lipschitz continuous 
on  each of the intervals $(\xi_0,\xi_1),\ldots,(\xi_k, \xi_{k+1})$ and has one-sided limits $(G_{|\R\setminus\Theta})''(\xi_i-)$ and $(G_{|\R\setminus\Theta})''(\xi_i+)$ for all $i\in\{1, \ldots, k\}$.
 Moreover, one can show that  for all $i\in\{1,\ldots,k\}$,
\begin{equation}\label{rr4}
(G'\cdot \mu+\tfrac{1}{2}G''\cdot\sigma^2)(\xi_i+) = (G'\cdot \mu+\tfrac{1}{2}G''\cdot\sigma^2)(\xi_i-).
\end{equation}

 By a slight abuse of notation we define an extension $G''\colon \R\to \R$ by  
taking 
\begin{equation}\label{r12}
G''(\xi_i)=(G_{|\R\setminus\Theta})''(\xi_i+)+\frac{2G'(\xi_i)\cdot (\mu(\xi_i+)-\mu(\xi_i))}{\sigma^2(\xi_i)}
\end{equation}
for $i\in\{1, \ldots, k\}$. Clearly, $G''$ is then a bounded Lebesgue-density of $G'$. Furthermore, it is straightforward to check that $\widetilde\mu$ and $\widetilde\sigma$ are Lipschitz continuous, which  completes the proof of the lemma.     
\end{proof}

Next, choose $G$ according to Lemma~\ref{transform1} and define a stochastic process $Z\colon [0,1]\times\Omega\to\R$ by
\begin{equation}\label{tr}
Z_t=G(X_t), \quad t\in[0,1].
\end{equation}

\begin{lemma}
The process $Z$ is the unique strong solution of the SDE
\begin{equation}\label{sde1}
\begin{aligned}
dZ_t & = \widetilde\mu(Z_t) \, dt + \widetilde\sigma(Z_t) \, dW_t, \quad t\in [0,1],\\
Z_0 & = G(x_0)
\end{aligned}
\end{equation}
with $\widetilde  \mu$ and $\widetilde \sigma$ according to Lemma~\ref{transform1}(ii).
\end{lemma}

\begin{proof} According to Lemma~\ref{transform1}(ii), $G'$ is absolutely continuous.  We may therefore apply It\^{o}'s formula, see e.g.~\cite[Problem 3.7.3]{ks91}, to conclude that for every $t\in[0,1]$ we have $\PP\text{-a.s.}$,
\[
G(X_t)=G(x_0)+\int_0^t(G'(X_s)\cdot \mu(X_s)+\tfrac{1}{2}G''(X_s)\cdot\sigma^2(X_s))\,ds+\int_0^t G'(X_s)\cdot\sigma(X_s)\,dW_s,
\]
which implies that $Z$ is a strong solution of the SDE \eqref{sde1}. Due to the Lipschitz continuity of $\widetilde\mu$ and $\widetilde\sigma$, see Lemma \ref{transform1}(ii),
  the strong solution of \eqref{sde1} is unique.
\end{proof}

For every $n\in\N$ we use $\eultr_{n}=(\eultr_{n,t})_{t\in[0,1]}$ to denote the time-continuous Euler-Maruyama scheme  with step-size $1/n$ associated to the SDE \eqref{sde1}, i.e. 
$\eultr_{n,0}=G(x_0)$ and
\[
\eultr_{n,t}=\eultr_{n,i/n}+\widetilde\mu(\eultr_{n,i/n})\cdot (t-i/n)+\widetilde\sigma(\eultr_{n,i/n})\cdot (W_t-W_{i/n})
\]
for $t\in(i/n,(i+1)/n]$ and $i\in\{0,\ldots,n-1\}$. The following estimates are standard error bounds for the time-continuous Euler-Maruyama scheme associated to an SDE with Lipschitz continuous coefficients.

\begin{lemma}\label{treulprop}
Let $p\in [1,\infty)$. Then there exists  $c\in(0, \infty)$ such that for all $n\in\N$,
\begin{itemize}
\item[(i)]
$
\displaystyle{\EE\bigl[\| \eultr_{n}\|_\infty^p\bigr]\leq c}
$,\\[-.25cm]
\item[(ii)] 
$
\displaystyle{
\bigl(\EE\bigl[\|Z-\eultr_{n})\|_\infty^p\bigr]\bigr)^{1/p}\leq c/\sqrt{n}}
$.
\end{itemize}
\end{lemma}

Finally, we provide an estimate for the transformed time-continuous Euler-Maruyama scheme $G\circ \eul_n = (G(\eul_{n,t}))_{t\in[0,1]}$.

\begin{lemma}\label{treulprop2}
Let $p\in [1,\infty)$. Then there exists  $c\in(0, \infty)$ such that for all $n\in\N$,
\[
\EE\bigl[\|G\circ \eul_{n}\|_\infty^p\bigr]\leq c.
\]
\end{lemma}

\begin{proof} 
According to Lemma~\ref{transform1}(i), $G$ is Lipschitz continous and hence satisfies a linear growth condition, i.e. there exists $c\in (0, \infty)$ such that $|G(x)|\leq c\cdot (1+|x|)$ for all $x\in\R$. Hence
\[
\|G\circ \eul_{n}\|_\infty \leq c\cdot(1+\|\eul_{n}\|_\infty),
\]
which jointly with Lemma \ref{eulprop}  implies the statement of the lemma.
\end{proof}

\subsection{Proof of Theorem~\ref{Thm1}}\label{4.4}
We choose $G$ 
and a Lebesgue density $G''$ of $G$
according to Lemma~\ref{transform1}, define $Z$ by~\eqref{tr},   
and for every $n\in\N$ we define a function $u_n\colon [0,1]\to[0, \infty)$ by
\[
u_n(t)=\EE\bigl[\sup_{s\in[0, t]} |G(\eul_{n,s})-\eultr_{n,s}|^p\bigr].
\]
Note that the functions $u_n$, $n\in\N$, are well-defined 
and bounded
due to Lemma~\ref{treulprop}(i) and Lemma~\ref{treulprop2}. 
 
Below we show that there exists  $c\in(0, \infty)$ such that for all $n\in\N$ 
and all $t\in[0,1]$,
\begin{equation}\label{G7}
u_n(t)\leq c\cdot\Bigl(\frac{1}{n^{p/2}}+\sum_{i=1}^k \EE\Bigl[\Bigl |\int_0^1  \ind_{\{(\eul_{n,s}-\xi_i)\cdot(\eul_{n, \usn}-\xi_i)\leq 0\}}\,ds\Bigr|^p\Bigr]+\int_0^t u_n(s)\, ds\Bigr).
\end{equation}
Using Proposition~\ref{prop1} we conclude from~\eqref{G7} that there exists $c\in (0,\infty)$ such that for all $n\in\N$ and all $t\in[0,1]$,
\[
u_n(t)\leq c\cdot \Bigl( \frac{1}{n^{p/2}} +\int_0^t u_n(s)\, ds\Bigr).
\]
By Gronwall's inequality it then follows that there exists $c\in (0,\infty)$ such that for all $n\in\N$,
\begin{equation}\label{p1}
u_n(1) \le \frac{c}{n^{p/2}}.
\end{equation} 
Using the fact that $G^{-1}$ is Lipschitz continuous, see Lemma~\ref{transform1}(i), as well as Lemma~\ref{treulprop}(ii) and~\eqref{p1} we conclude that there exist $c_1,c_2\in (0,\infty)$ such that for all $n\in\N$,  
\[
\EE\bigl[\|X-\eul_n\|_\infty^p\bigr]  \le c_1\cdot \EE\bigl[\|Z-G\circ \eul_n\|_\infty^p\bigr] \le 2^p\cdot c_1\cdot\bigl( \EE\bigl[\|Z-\widehat Z_n\|_\infty^p\bigr] +u_n(1)\bigr)\le \frac{c_2}{n^{p/2}}, 
\]
which yields the statement of Theorem~\ref{Thm1}.

It remains to prove~\eqref{G7}. Let $n\in\N$. Clearly, for every $t\in[0,1]$, 
\[
\eultr_{n,t}=G(x_0)+\int_0^t \widetilde\mu(\eultr_{n,\usn})\, ds+\int_0^t\widetilde\sigma(\eultr_{n,\usn})\,dW_s.
\]
Since $G'$ is absolutely continuous, see Lemma~\ref{transform1}(ii), we may apply It\^{o}'s formula, see e.g.~\cite[Problem 3.7.3]{ks91}, to obtain that $\PP\text{-a.s.}$ for all $t\in[0,1]$,
\begin{align*}
G(\eul_{n,t}) & =G(x_0)+\int_0^t(G'(\eul_{n,s})\cdot \mu(\eul_{n,\usn})+\tfrac{1}{2}G''(\eul_{n,s})\cdot\sigma^2(\eul_{n,\usn}))\,ds\\
& \qquad +\int_0^t G'(\eul_{n,s})\cdot\sigma(\eul_{n,\usn})\,dW_s\\
& = G(x_0)+\int_0^t \widetilde\mu(G(\eul_{n,\usn}))\, ds
+ \int_0^t \bigl(G'(\eul_{n,s})-G'(\eul_{n,\usn})\bigr)\cdot \mu(\eul_{n,\usn})\,ds \\
& \qquad +\int_0^t\widetilde\sigma(G(\eul_{n,\usn}))\,dW_s 
+  \int_0^t \bigl(G'(\eul_{s,n})-G'(\eul_{n,\usn})\bigr)\cdot \sigma(\eul_{n,\usn})\,dW_s\\
& \qquad + \frac{1}{2}\cdot \int_0^t \bigl(G''(\eul_{n,s})-G''(\eul_{n,\usn})\bigr)\cdot \sigma^2(\eul_{n,\usn})\,ds. 
\end{align*}
It follows that $\PP\text{-a.s.}$ for all $t\in[0,1]$,
\[
G(\eul_{n,t})-\eultr_{n,t}=\sum_{i=1}^3 V_{n,i, t},
\]
where 
\begin{align*}
V_{n,1,t} & = \int_0^t (\widetilde\mu(G(\eul_{n,\usn}))- \widetilde \mu(\eultr_{n,\usn}))\, ds+ \int_0^t(\widetilde\sigma(G(\eul_{n,\usn}))- \widetilde \sigma(\eultr_{n,\usn}))\,dW_s,\quad \\
V_{n,2,t}&=\int_0^t \bigl(G'(\eul_{n,s})-G'(\eul_{n,\usn})\bigr)\cdot \mu(\eul_{n,\usn})\,ds 
+  \int_0^t \bigl(G'(\eul_{n,s})-G'(\eul_{n,\usn})\bigr)\cdot \sigma(\eul_{n,\usn})\,dW_s,\\
V_{n,3,t} & = \frac{1}{2}\cdot \int_0^t \bigl(G''(\eul_{n,s})-G''(\eul_{n,\usn})\bigr)\cdot \sigma^2(\eul_{n,\usn})\,ds. 
\end{align*}
Hence, for all $t\in[0,1]$,
\begin{equation}\label{jj}
u_n(t)\leq 3^p\cdot \sum_{i=1}^3 \EE\bigl[\sup_{s\in[0,t]} |V_{n,i,s}|^p\bigr].
\end{equation}

We next estimate the single summands on the right hand side of \eqref{jj}. 
Using the H\"older inequality, the Burkholder-Davis-Gundy inequality and the
Lipschitz continuity of $\widetilde\mu$ and $\widetilde\sigma$, see Lemma \ref{transform1}(ii), 
we obtain that there exists
 $c\in(0, \infty)$ such that for all $n\in\N$ and all $t\in[0,1]$,
\begin{equation}\label{t3}
\EE\bigl[\sup_{s\in[0,t]} |V_{n,1,s}|^p\bigr]\leq c\cdot \int_0^t \EE\bigl[|G(\eul_{n,\usn}))-\eultr_{n,\usn}|^p\bigr]\, ds\leq c\cdot \int_0^t u_n(s)\, ds.
\end{equation}
Furthermore, using the H\"older inequality, the Burkholder-Davis-Gundy inequality as well as the Lipschitz continuity of $G'$, see Lemma \ref{transform1}(ii), and employing \eqref{LG} as well as  Lemma~\ref{eulprop} we conclude that there exist
 $c_1, c_2,c_3\in(0, \infty)$ such that for all $n\in\N$ and all $t\in[0,1]$,
\begin{equation}\label{t4}
\begin{aligned}
\EE\bigl[\sup_{s\in[0,t]} |V_{n,2,s}|^p\bigr]&\leq c_1\cdot \int_0^t \EE\bigl[|G'(\eul_{n,s}))-G'(\eul_{n,\usn})|^p\cdot (|\mu(\eul_{n,\usn})|^p+|\sigma(\eul_{n,\usn})|^p\bigr)\bigr]\, ds\\
&\leq c_2\cdot \int_0^t \bigl(\EE\bigl[|\eul_{n,s}-\eul_{n,\usn}|^{2p}\bigr]\bigr)^{1/2}\cdot\bigl(1+ \EE\bigl[|\eul_{n,\usn}|^{2p}\bigr]\bigr)^{1/2}\, ds\leq \frac{c_3}{n^{p/2}}.
\end{aligned}
\end{equation}

For estimating  $\EE\bigl[\sup_{s\in[0,t]} |V_{n,3,s}|^p\bigr]$ we put
\[
B= \Bigl(\bigcup_{i=1}^{k+1} (\xi_{i-1}, \xi_{i})^2\Bigr)^c
\] 
and we note that 
$B=\bigcup_{i=1}^k \{(x,y)\in\R^2:  (x-\xi_{i})\cdot(y-\xi_{i})\leq 0\}$.
Using Lemma~\ref{transform1}(ii) and~\eqref{LG} 
we obtain 
that there exists
 $ c\in(0, \infty)$ such that for all $x,y\in\R$,
\begin{align*}
&|G''(x)\cdot\sigma^2(y)-G''(x)\cdot\sigma^2(y)|\leq \begin{cases}
c\cdot (1+y^2)\cdot|x-y|, & (x,y)\in B^c, \\
c\cdot (1+y^2),&(x,y)\in B.
\end{cases}
\end{align*}
Hence  
there exists
 $ c\in(0, \infty)$ such that
for all $t\in[0,1]$,
\begin{equation}\label{end1}
\begin{aligned}
 \sup_{s\in[0,t]} |V_{n,3,s}|^p 
& \leq c\cdot \Bigl(\Bigl|\int_0^t(1+\eul_{n,\usn}^2)\cdot |\eul_{n,s}-\eul_{n,\usn}|\, ds \Bigr|^p \\
& \qquad\qquad + \Bigl| \int_0^t(1+\eul_{n,\usn}^2)\cdot 1_{\{(\eul_{n,s}, \eul_{n,\usn})\in B\}}\, ds\Bigr|^p\Bigr).
\end{aligned}
\end{equation}
Using Lemma~\ref{eulprop} we obtain as in~\eqref{t4}
that there exists $c\in (0,\infty)$ such that for all $t\in[0,1]$,
\begin{equation}\label{t2}
\begin{aligned}
\EE \Bigl[\Bigl|\int_0^t(1+\eul_{n,\usn}^2)\cdot |\eul_{n,s}-\eul_{n,\usn}|\, ds \Bigr|^p\Bigr] & \le \frac{c}{n^{p/2}}.
\end{aligned}
\end{equation}
Furthermore, for all $i\in\{1, \ldots, k\}$ and all $s\in[0,1]$,
\begin{align*}
|\eul_{n,\usn}|\cdot 1_{\{(\eul_{n,s}-\xi_i)\cdot (\eul_{n,\usn}-\xi_i)\leq 0\}} & \leq (|\xi_i|+|\eul_{n,\usn}-\xi_i|)\cdot 1_{\{(\eul_{n,s}-\xi_i)\cdot (\eul_{n,\usn}-\xi_i)\leq 0\}}\\ & \le
(|\xi_i|+|\eul_{n,\usn}-\eul_{n,s}|)\cdot \ind_{\{(\eul_{n,s}-\xi_i)\cdot (\eul_{n,\usn}-\xi_i)\leq 0\}},
\end{align*}
which yields that for all $s\in[0,1]$,
\[
(1+\eul_{n,\usn}^2)\cdot 1_{\{(\eul_{n,s}, \eul_{n,\usn})\in B\}} \le (1+2\max_{i=1,\dots,k}\xi_i^2)\cdot \sum_{i=1}^k \ind_{\{(\eul_{n,s}-\xi_i)\cdot (\eul_{n,\usn}-\xi_i)\leq 0\}} + 2 (\eul_{n,\usn}-\eul_{n,s})^2.
\]
By the latter inequality and Lemma~\ref{eulprop} we conclude that there exists $c\in (0,\infty)$ such that for all $t\in[0,1]$, 
\begin{equation}\label{end2}
\begin{aligned}
& \EE\Bigl[\Bigl| \int_0^t(1+\eul_{n,\usn}^2)\cdot 1_{\{(\eul_{n,s}, \eul_{n,\usn})\in B\}}\, ds\Bigr|^p\Bigr]\\ & \qquad\qquad\le c\cdot  \sum_{i=1}^k \EE\Bigl[\Bigl|\int_0^t\ind_{\{(\eul_{n,s}-\xi_i)\cdot (\eul_{n,\usn}-\xi_i)\leq 0\}}\, ds \Bigr|^p\Bigr]
 + \frac{c}{n^{p}}.
\end{aligned}
\end{equation}
Combining~\eqref{end1},~\eqref{t2} and~\eqref{end2} we see that there exists $c\in (0,\infty)$ such that for all $t\in[0,1]$,
\[
\EE\bigl[\sup_{s\in[0,t]} |V_{n,3,s}|^p\bigr] \le  \frac{c}{n^{p/2}} + c\cdot
 \sum_{i=1}^k \EE\Bigl[\Bigl|\int_0^t\ind_{\{(\eul_{n,s}-\xi_i)\cdot (\eul_{n,\usn}-\xi_i)\leq 0\}}\, ds \Bigr|^p\Bigr], 
\]
which jointly with~\eqref{jj}, \eqref{t3} and~\eqref{t4} yields the estimate~\eqref{G7} and hereby completes the proof of Theorem~\ref{Thm1}.

\subsection{Proof of Theorem~\ref{Thm2}}\label{4.5}
Clearly, for all $n\in\N$, 
\begin{equation}\label{two0}
   \bigl(\EE\bigl[\|X-\overline X_n\|_q^p\bigr]\bigr)^{1/p} 
\le \bigl(\EE\bigl[\|X-\eul_n\|_q^p\bigr]\bigr)^{1/p} +\bigl(\EE\bigl[\|\eul_n-\overline X_n\|_q^p\bigr]\bigr)^{1/p}. 
\end{equation}
Moreover, by Theorem~\ref{Thm1} there exists $c\in (0,\infty)$ such that for all $n\in\N$,
\begin{equation}\label{two1}
\bigl(\EE\bigl[\|X-\eul_n\|_q^p\bigr]\bigr)^{1/p} \le \bigl(\EE\bigl[\|X-\eul_n\|_\infty^p\bigr]\bigr)^{1/p} \le c/\sqrt{n}.
\end{equation}

For
$n\in\N$ 
define a stochastic process  
$\overline W_n = (\overline W_{n,t})_{t\in[0,1]}$ by
\[
\overline W_{n,t} = (n\cdot t-i)\cdot W_{n,(i+1)/n} + (i+1-n\cdot t)\cdot  W_{n,i/n}
\]
for $t\in [i/n,(i+1)/n]$
 and $i\in\{0,\ldots,n-1\}$. 
Then for every $r\in [1,\infty)$
there exists $c\in (0,\infty)$ such that for all $n\in\N$,
\begin{equation}\label{two01}
\bigl(\EE\bigl[\|W-\overline W_n\|_q^r\bigr]\bigr)^{1/r}\leq \begin{cases} 
c/\sqrt n, & \text{ if }q < \infty,\\ c\sqrt{\ln (n+1)}/\sqrt{n}, & \text{ if }q = \infty,
\end{cases} 
\end{equation}
see, e.g.~\cite{Speckman79} for the case $q\in[1, \infty)$ and~\cite{Faure92} for the case $q=\infty$.

Note that for all $n\in\N$ and all $t\in[0,1]$, 
\begin{align*}
|\eul_{n,t} - \overline X_{n,t}| & = 
\Bigl|\sum_{i=0}^{n-1}\sigma(\eul_{n,i/n})\cdot \ind_{[i/n,(i+1)/n]}(t)\cdot (W_t- \overline W_{n,t})\Bigr|\\
& \le \sup_{s\in[0,1]} |\sigma(\eul_{n,s})|\cdot |W_t-\overline W_{n,t}|.
\end{align*}
Hence, by~\eqref{LG} and Lemma~\ref{eulprop} there exist $c_1,c_2\in (0,\infty)$ such that for all $n\in\N$,
\begin{align*}
\bigl(\EE\bigl[\|\eul_n-\overline X_n\|_q^p\bigr]\bigr)^{1/p} &\le c_1 \cdot \bigl(1 + \bigl(\EE\bigl[\|\eul_n\|_\infty^{2p}\bigr]\bigr)^{1/(2p)}\bigr)\cdot \bigl(\EE\bigl[\|W-\overline W_n\|_q^{2p}\bigr]\bigr)^{1/{(2p)}}\\
&\le c_2\cdot \bigl(\EE\bigl[\|W-\overline W_n\|_q^{2p}\bigr]\bigr)^{1/{(2p)}},
\end{align*}
which jointly with~\eqref{two01}, \eqref{two1} and~\eqref{two0} completes the proof of the theorem.

\bibliographystyle{acm}
\bibliography{bibfile}

\end{document}